\documentclass[12pt,reqno]{amsart}

\usepackage{amsmath}
\usepackage{amssymb}
\usepackage{amscd}
\usepackage{amsfonts}

\newtheorem{theorem}{Theorem}[section]
\newtheorem{lemma}[theorem]{Lemma}
\newtheorem{proposition}[theorem]{Proposition}

\newtheorem{remark}[theorem]{Remark}
\newtheorem{corollary}[theorem]{Corollary}

\numberwithin{equation}{section}

\begin{document}
\title[]{ Aronson-B\'enilan estimates for the porous medium equation under the Ricci flow}

\author{Huai-Dong Cao and Meng Zhu}

\address{Huai-Dong Cao \\ Department of Mathematics\\ University of Macau\\Macao, China \& Department of Mathematics\\Lehigh University \\Bethlehem, PA 18015, USA} \email{huc2@lehigh.edu}

\address{Meng Zhu\\ Shanghai Center for Mathematical Sciences\\Fudan University\\Shanghai, 200433, China} \email{mengzhu@fudan.edu.cn}
\date{}
\thanks{The first author was partially supported by NSF Grant DMS-0909581 and by FDCT/016/2013/A1.
The second author is partially supported by China Postdoctral Science Foundation Grant No. 2013M531105}
\maketitle

\begin{abstract}
In this paper we study the porous medium equation (PME) coupled with the Ricci flow on complete manifolds with bounded nonnegative curvature. In particular, we derive Aronson-B\'enilan and Li-Yau-Hamilton type differential Harnack estimates for positive solutions  to the PME, with a linear forcing term, under the Ricci flow.
\end{abstract}

\textbf{Keywords:} Porous medium equation, Ricci flow, Aronson-B\'enilan estimate, Harnack inequality

\section{Introduction}

Differential equations of the form
\begin{equation}\label{eq1}
\frac{\partial u}{\partial t}=\Delta u^{p}
\end{equation}
are of great interest due to their importance in mathematics, physics, and applications in many other fields. For $p=1$ it is the well-known heat equation. When $p>1$, equation \eqref{eq1} is known as the porous medium equation (PME), which models the flow of gas through porous medium, ground water filtration, heat radiation in plasmas, etc. We refer the readers to [18] for basic theory and various applications of the porous medium equation, and \cite{DH} concerning the regularity of the free boundary
for the porous medium equation, in the Euclidean space. In the case where $p<1$, equation \eqref{eq1} becomes the so-called fast diffusion equation (FDE).

In 1979, Aronson and B\'enilan \cite{ArBe1979} obtained an important second order differential inequality
\begin{equation}
\sum_i\frac{\partial}{\partial x_i}(pu^{p-2}\frac{\partial u}{\partial x_i})\geq-\frac{\kappa}{t},
\end{equation}
where $\kappa=\frac{n}{2+n(p-1)}$, for any positive solution $u$ to \eqref{eq1} on the Euclidean space ${\mathbb R}^n$ with $p>(1-\frac{2}{n})^+$.
On the other hand, for any complete Riemannian manifold $(M^n, g_{ij})$ with Ricci curvature bounded from below by $-K$ for some $K\geq0$, Li and Yau \cite{LiYa1986} discovered the following celebrated differential Harnack estimate, now widely called the Li-Yau estimate,
\begin{equation}\label{Li-Yau}
\frac{|\nabla u|^2}{u^2} - \alpha\frac{u_t}{u} \leq \frac{n\alpha^2K}{2(\alpha-1)}+\frac{n\alpha^2}{2t},
\end{equation}
where $\alpha>1$, for any positive solution $u$ to the heat equation on $M^n$. Moreover, in the special case where $K=0$, there holds the sharp Li-Yau estimate
\begin{equation}
\frac{|\nabla u|^2}{u^2} - \frac{u_t}{u} \leq \frac{n}{2t}.
\end{equation}
In the same paper, Li and Yau showed that significant results can be obtained by applying the differential Harnack estimates above, such as the classical parabolic Harnack inequality, Gaussian upper and lower bounds of the heat kernel, estimates of the Green's function, and estimates of the Dirichlet and Neumann eigenvalues of the Laplace operator.

Subsequently, matrix differential Harnack estimates for the heat equation were proved by Hamilton \cite{Ham1993cag} on Riemannian manifolds, and by Ni and the first author \cite{CaNi2005} on K\"ahler manifolds.

Unlike the study of the heat equation, PME on Riemannian manifolds were only investigated recently. In \cite{Vaz2007}, Aronson-B\'enilan and Li-Yau type estimate was first proved by V\'azquez for positive solutions to PME on complete manifolds with nonnegative Ricci curvature. More recently, under the weaker assumption that the Ricci curvature of $M$ is bounded from below by $-K$ for some $K\ge 0$, Lu, Ni, V\'azquez and Villani \cite{LNVV2009} derived more general Aronson-B\'enilan and Li-Yau type gradient estimates for PME and some FDE. In particular, for any $\alpha >1$ and any smooth bounded positive solution $u$ to PME, they proved
\begin{equation}\label{eqNi}
\alpha\frac{v_t}{v}-\frac{|\nabla v|^2}{v}\geq -(p-1)\kappa\alpha^2\left[\frac{1}{t}+\frac{(p-1)}{\alpha-1} K v_{max}\right],
\end{equation}
where $v=\frac{p}{p-1}u^{p-1}$, $\kappa=\frac{n}{2+n(p-1)}$,  and $v_{max}=\max_{M\times[0,T]}v$.  See also the related work in \cite{HHL2013}.

The study of differential Harnack estimates for parabolic equations under the Ricci flow
\begin{equation}\label{rf}
\frac{\partial g_{ij}}{\partial t}=-2R_{ij}
\end{equation}
on complete manifolds was originated from Hamilton's work. In \cite{Ham1988}, for the Ricci flow on Riemann surfaces with positive curvature, he showed a Li-Yau type inequality for the evolving scalar curvature of the Ricci flow.\footnote{Later, Chow \cite{Cho1991} was able to remove the positivity assumption and obtained a similar estimate.} In higher dimensions, both matrix and trace differential Harnack estimates, also known as Li-Yau-Hamilton estimates, were obtained by Hamilton \cite{Ham1993jdg} for the Ricci flow with bounded and nonnegative curvature operator.   These estimates played a crucial role in the study of singularity formations of the Ricci flow on three manifolds and the Hamilton-Perelman solution to the Poincar\'e conjecture. We remark that
Brendle \cite{Brendle} has generalized Hamilton's differential Harnack inequality under the weaker curvature assumption
of  $(M^n, g(t)) \times {\mathbb R}^2$ having bounded and nonnegative isotropic curvature. The Li-Yau-Hamilton estimate for the K\"ahler-Ricci flow with nonnegative holomorphic bisectional curvature was obtained by the first author \cite{Cao1992}. In addition, in \cite{Per2002a} Perelman obtained a Li-Yau type estimate for the fundamental solution of the conjugate heat equation under the Ricci flow (see also \cite{Ni2006}).

Recently, by using Hamilton's trace differential Harnack estimate \cite{Ham1993jdg}, X. Cao and Hamilton \cite{CaHa2009} proved the Li-Yau-Hamilton type estimate
\begin{equation}\label{Cao-Hamilton}
2\Delta v-|\nabla v|^2-3R-\frac{2n}{t}\leq0
\end{equation}
on compact manifolds, where $v=-\ln u$ and $u$ is any positive solution to the forward conjugate heat equation
\begin{equation}\label{conjugate equation}\frac{\partial u}{\partial t}=\Delta u+Ru\end{equation}
under the Ricci flow with nonnegative curvature operator. Also, certain matrix differential Harnack estimate was shown by Ni \cite{Ni2007} for positive solutions of \eqref{conjugate equation} under the K\"ahler-Ricci flow.


In this paper, we consider equation of type \eqref{eq1} with  a linear forcing term,
\begin{equation}\label{eq2}
\frac{\partial u}{\partial t}=\Delta u^{p}+ Ru,
\end{equation}
under the Ricci flow on a complete manifold $M^n$  and prove Aronson-B\'enilan and
Li-Yau-Hamilton type differential Harnack estimates for positive solutions to the PME (i.e., \eqref{eq2} with $ p>1$). Here $\Delta=\Delta_{g(t)}$ is the Laplace operator with respect to the evolving metric $g(t)$ of the Ricci flow \eqref{rf}, and $R=R(x, t)$ is the scalar curvature of $g(t)$. One may check directly that, thanks to the term $Ru$, any smooth solution $u$ of \eqref{eq2} satisfies
$$\frac{d}{dt}\int_M u d\mu=0,$$
because the volume form evolves under the equation $d(d\mu)/dt=-R d\mu$.
Note that  when $p=1$, \eqref{eq2} is simply the conjugate equation \eqref{conjugate equation} to the backward heat equation.

When $p>1$, one usually expresses the PME \eqref{eq1} in the following form:
\begin{align}\label{eq1.1}
\frac{\partial u}{\partial t}=\nabla(u\nabla v)=(p-1)v\Delta u+(p-1)\nabla u\cdot\nabla v,
\end{align}
where $v=\frac{p}{p-1}u^{p-1}$. It is clear that if $u$ is a positive solution to \eqref{eq1}, then \eqref{eq1.1} is a parabolic equation, hence the regularity can be studied.

Similarly, with $v=\frac{p}{p-1}u^{p-1}$, \eqref{eq2} can be rewritten as
\begin{align}\label{eq1.2}
\frac{\partial u}{\partial t}=(p-1)v\Delta u+ (p-1)\nabla u\cdot \nabla v+Ru,
\end{align}
which can still be considered formally as the conjugate equation to \eqref{eq1.1} under the Ricci flow if we fix the potential function $v$.
Indeed, one can see that for any smooth functions $u$ and $w$
\begin{align*}
\frac{d}{dt}\int_M uwd\mu &= \int_M (u_tw + uw_t-Ruw) d\mu \\
&= \int_M \left[u_t+(p-1)v\Delta u+(p-1)\nabla u\nabla v\right]wd\mu\\
&\quad\  +\int_M u\left[w_t-(p-1)v\Delta w-(p-1)\nabla w\nabla v - Rw\right]d\mu.
\end{align*}

Our main result in this paper is the following Aronson-B\'enilan and Li-Yau-Hamilton type estimate for smooth positive solutions to the PME \eqref{eq2}:

\begin{theorem}\label{theorem 1}
Let $(M^n, g_{ij}(t))$, $t\in[0, T)$, be a complete solution to the Ricci flow with bounded and nonnegative curvature operator at each time $t$. If $u$ is a bounded smooth positive solution to \eqref{eq2} with $p>1$ and $v=\frac{p}{p-1}u^{p-1}$, then we have
$$\frac{|\nabla v|^2}{v}-2\frac{v_t}{v}-\frac{R}{v}-\frac{d}{t}\leq 0$$
on $M\times(0,T)$, where $d=\max\{2\alpha, 1\}$ and $\alpha=\frac{n(p-1)}{1+n(p-1)}$.
\end{theorem}

As immediate consequences of Theorem \ref{theorem 1}, we have the following two versions of Harnack inequalities.

\begin{corollary}\label{cor 1.1} Under the same assumptions as in Theorem \ref{theorem 1},  if $v_{min}=\inf_{M\times[0,T)}v>0$, then for any points $x_1$, $ x_2\in M$ and $0<t_1<t_2<T$, we have
\begin{equation}\label{eq cor1}
v(x_1, t_1) \leq v(x_2, t_2) (\frac{t_2}{t_1})^{d/2}\exp\left(\frac{\Gamma}{2v_{min}}\right),
\end{equation}
where  $d$ is the constant in Theorem \ref{theorem 1}, and $\Gamma=\inf_{\gamma}\int_{t_1}^{t_2}(R+\left|\frac{d\gamma}{d\tau}\right|^2_{g_{ij}(\tau)})d\tau$ with the infimum taking over all smooth curves $\gamma(\tau)$ in $M$, $t_1\le \tau\le t_2$, so that $\gamma(t_1)=x_1$ and $\gamma(t_2)=x_2$.

\end{corollary}

\begin{corollary}\label{cor 1.2}
Under the same assumptions as in Theorem \ref{theorem 1},  for any points $x_1$, $ x_2\in M$ and $0<t_1<t_2<T$, we have
\begin{equation}\label{eq cor2}
v(x_1, t_1)-v(x_2, t_2)\leq \frac{d}{2} v_{max} \ln \frac{t_2}{t_1} + \frac{1}{2} R_{max} (t_2-t_1)+\frac{1}{2}\frac{d_{t_1}^2(x_1,x_2)}{t_2-t_1}
\end{equation}
where $d$ is the constant in Theorem \ref{theorem 1}, $v_{max}=\sup_{M\times[0,T)}v$, $R_{max}=\sup_{M\times[0,T)}R$, and $d_{t_1} (x_1,x_2)$ is the distance between $x_1, x_2$ at time $t_1$ with respect to the metric $g_{ij}(t_1)$.
\end{corollary}

More generally, for PME \eqref{eq2} with $p>1$, we actually have

\begin{theorem}\label{harnack PME}
Let $(M^n, g_{ij}(t))$, $t\in[0, T]$, be a complete solution to the Ricci flow with bounded and nonnegative curvature operator. If $u$ is a bounded smooth positive solution to \eqref{eq2} with $p>1$, then for $v=\frac{p}{p-1}u^{p-1}$ and any $b\in[1,\infty)$, we have
$$\frac{|\nabla v|^2}{v}-b\frac{v_t}{v}-(b-1)\frac{R}{v}-\frac{d}{t}\leq C_0|b-2|R_{max}$$
on $M\times(0,T]$, where $\alpha=\frac{bn(p-1)}{2+bn(p-1)}$, $d=\max\{b\alpha, \frac{b}{2}\}$, $R_{max}=\sup_{M\times[0,T]}R$, and
\begin{equation*}
C_0=\left\{\begin{array}{ll}
\frac{2\alpha}{n}+\sqrt{\frac{b\alpha(p-1)}{2}}, & if\ b\geq2\\
\sqrt{\frac{b\alpha(p-1)(n-1)}{2n}}, & if\ 1\leq b<2.
\end{array}\right.
\end{equation*}
\end{theorem}

\begin{remark}
By the work of S. Brendle in \cite{Brendle} on the Li-Yau-Hamilton estimate for the Ricci flow, all of the above results remain valid if we replace the assumption of nonnegative curvature operator by the weaker condition that $(M^n, g(t)) \times {\mathbb R}^2$ has nonnegative isotropic curvature at each time $t$.
\end{remark}

\begin{remark}
We point out that one can apply similar techniques used in the proof of Theorem 1.4 to obtain Aronson-B\'enilan and Li-Yau-Hamilton type estimates for the FDE case, except the calculations are somewhat more involved. Moreover, by applying the same techniques, one can extend the Li-Yau-Hamilton estimate \eqref{Cao-Hamilton} of X. Cao and R. Hamilton \cite{CaHa2009} to the complete noncompact setting. These will be treated elsewhere.
\end{remark}

This paper is organized as follows. In Section 2, we introduce the main quantity, which is involved in our differential Harnack estimates for PME, and compute its evolution under the Ricci flow. In Section 3, we prove Theorem \ref{harnack PME} for positive solutions to the PME coupled with the Ricci flow  and derive the two Harnack inequalities stated in Corollary 1.2 and Corollary 1.3, respectively. \\

\hspace{-.4cm}\textbf{Acknowledgements}: The first author would like to thank the University of Macau, where part of the work was carried out during his visit in Spring of 2014, for its hospitality and the support by RDG010 grant. The second author is grateful to Professors Qing Ding, Jixiang Fu, Jiaxing Hong, Jun Li, Quanshui Wu, and Weiping Zhang for their constant support and encouragement. He also wants to thank the Shanghai Center for Mathematical Sciences, where the work was carried out, for its hospitality and support.

\section{The Main Quantity for PME and FDE}

In this section, we compute the evolution of certain quantity $F$ introduced below which is essential for deriving our Aronson-B\'enilan and Li-Yau-Hamilton estimates in the next section.
From now on let  $g_{ij}(t)$ be a solution to the Ricci flow \eqref{rf} on $M^n\times [0, T)$ such that $(M^n, g_{ij}(t))$ is a complete Riemannian manifold for each $t\in [0, T)$. Consider the equation
\begin{equation}\label{eq3}
\frac{\partial u}{\partial t}=\Delta u^p + aRu,
\end{equation}
where $a$ is an arbitrary constant, $p\in(0,1)\cup(1,\infty)$, and the Laplacian is with respect to the evolving metric $g_{ij}(t)$. Obviously,
\eqref{eq3} reduces to \eqref{eq2} for $a=1$.

Let $u>0$ be a smooth positive solution to \eqref{eq3}, and set $v=\frac{p}{p-1}u^{p-1}$. It is straightforward to check that
\begin{equation}
\nabla v=pu^{p-2}\nabla u=(p-1)\frac{v}{u}\nabla u,
\end{equation}
and
\begin{equation}\label{eqv}
\frac{\partial v}{\partial t}=(p-1)v\Delta v + |\nabla v|^2 + a(p-1)Rv.
\end{equation}

Define
\begin{equation}
\begin{aligned}
F&= \frac{|\nabla v|^2}{v}-b\frac{v_t}{v} + c\frac{R}{v}\\
&= -b(p-1)\Delta v +(1-b)\frac{|\nabla v|^2}{v}-ab(p-1)R+c\frac{R}{v},
\end{aligned}
\end{equation}
where $b$ and $c$ are arbitrary constants, and $v_t=\frac{\partial v}{\partial t}$.

We also define operator $\mathcal{L}$ by
$$\mathcal{L}=\frac{\partial}{\partial t}-(p-1)v\Delta.$$ Recall that (cf. \cite{LNVV2009})
\begin{equation}\label{quotient}
\mathcal{L}(\frac{f}{g})=\frac{1}{g}\mathcal{L}(f)-\frac{f}{g^2}\mathcal{L}(g)+2(p-1)v\nabla_i(\frac{f}{g})\nabla_i(\ln g).
\end{equation}

Our main purpose in this section is to compute the evolution equation $\mathcal{L}(F)$ of $F$.

\begin{proposition}\label{evolutionF}
We have
\begin{equation}\label{eqevolutionF}
\begin{aligned}
\mathcal{L}(F)= &\  2p\nabla_i F \nabla_i v +\frac{1}{v}(c\frac{\partial R}{\partial t}-2c\nabla_i R\nabla_i v+2(1-b)R_{ij}\nabla_i v \nabla_j v)\\
&\quad\ -(p-1)\left[ (ab+c)\frac{\partial R}{\partial t} -2a\nabla_i R\nabla_i v + 2R_{ij}\nabla_i v\nabla_j v\right]\\
&\quad\ -2(p-1)|\nabla^2 v|^2 -2b(p-1)R_{ij}\nabla_i\nabla_j v+2c(p-1)|Rc|^2\\
&\quad\ -b(p-1)^2(\Delta v)^2 + 2(1-b)(p-1)\frac{|\nabla v|^2}{v}\Delta v - ab(p-1)^2R\Delta v\\
&\quad\ + a(1-b)(p-1)\frac{|\nabla v|^2}{v}R +(1-b)\frac{|\nabla v|^4}{v^2}+c\frac{|\nabla v|^2}{v^2}R - ac(p-1)\frac{R^2}{v}.
\end{aligned}
\end{equation}
\end{proposition}

\begin{proof} First of all, we compute
\begin{align*}
\frac{\partial }{\partial t}\Delta v &= 2R_{ij}\nabla_i \nabla_j v + \Delta \frac{\partial v}{\partial t}\\
&= 2R_{ij}\nabla_i\nabla_j v + (p-1)v\Delta^2 v+(p-1)(\Delta v)^2 + 2(p-1)\nabla_i(\Delta v)\nabla_i v\\
&\quad\ + \Delta|\nabla v|^2 + a(p-1)\Delta(Rv).
\end{align*}
Using the Bochner formula
$$\Delta|\nabla v|^2=2\nabla_i(\Delta v)\nabla_i v + 2|\nabla^2 v|^2 + 2R_{ij}\nabla_i v\nabla_j v,$$
one obtains that
\begin{equation}\label{evolution1}
\begin{aligned}
\mathcal{L}(\Delta v)&=2p\nabla_i(\Delta v) \nabla_i v + 2R_{ij}\nabla_i\nabla_j v + (p-1)(\Delta v)^2 + 2|\nabla^2 v|^2\\
&\quad\  + 2R_{ij}\nabla_i v\nabla_j v+ a(p-1)v\Delta R+2a(p-1)\nabla_i R\nabla_i v + a(p-1)R\Delta v.
\end{aligned}
\end{equation}
On the other hand, since
\begin{align*}
\mathcal{L}(|\nabla v|^2)&=2R_{ij}\nabla_i v\nabla_j v + 2\nabla_i \frac{\partial v}{\partial t}\nabla_i v - (p-1)v\Delta |\nabla v|^2\\
&=2R_{ij}\nabla_i v\nabla_j v + 2(p-1)|\nabla v|^2\Delta v +2\nabla_i|\nabla v|^2\nabla_i v + 2a(p-1)v\nabla_i R\nabla_i v\\
&\quad\ + 2a(p-1)R|\nabla v|^2-2(p-1)v|\nabla^2 v|^2 -2(p-1)vR_{ij}\nabla_iv\nabla_j v,
\end{align*}
it follows from \eqref{quotient} that
\begin{equation}\label{evolution2}
\begin{aligned}
\mathcal{L}(\frac{|\nabla v|^2}{v})
&= 2p\nabla_i \frac{|\nabla v|^2}{v}\nabla_i v+\frac{2}{v}R_{ij}\nabla_i v\nabla_j v + 2(p-1)\frac{|\nabla v|^2}{v}\Delta v  + \frac{|\nabla v|^4}{v^2}\\
&\quad \ + 2a(p-1)\nabla_i R\nabla_i v + a(p-1)\frac{|\nabla v|^2}{v}R-2(p-1)|\nabla^2 v|^2\\
&\quad\ -2(p-1)R_{ij}\nabla_i v\nabla_j v.
\end{aligned}
\end{equation}

Finally, from \eqref{eqv} and the evolution equation
\begin{equation}\label{evolution3}
\frac{\partial R}{\partial t}=\Delta R + 2|Rc|^2
\end{equation}
of the scalar curvature $R$ under the Ricci flow,  we have
\begin{equation}\label{evolution4}
\begin{aligned}
\mathcal{L}(\frac{R}{v})
&=2p\nabla_i\frac{R}{v}\nabla_i v + \frac{|\nabla v|^2}{v^2}R - \frac{2}{v}\nabla_i R\nabla_i v+ \frac{1}{v}\frac{\partial R}{\partial t} - a(p-1)\frac{R^2}{v}\\
&\quad\  -(p-1)\frac{\partial R}{\partial t} + 2(p-1)|Rc|^2.
\end{aligned}
\end{equation}
Now, evolution equation \eqref{eqevolutionF} can be obtained directly from \eqref{evolution1}, \eqref{evolution2}, \eqref{evolution3}, and \eqref{evolution4}.
\end{proof}

Next, let us take $a=1$ so that equation \eqref{eq3} reduces to our equation \eqref{eq2} and further set $c=1-b$. Then, we have
\begin{align}\label{f=y-bz}
F&=\frac{|\nabla v|^2}{v}-b\frac{v_t}{v}+(1-b)\frac{R}{v}\\
&=y-bz,
\end{align}
where
$$y=\frac{|\nabla v|^2}{v}+\frac{R}{v}\qquad \textrm{and} \qquad z=\frac{v_t}{v}+\frac{R}{v}.$$

Note that in the special case $b=1$ (hence $c=0$), we have
$$y-z=-(p-1)\Delta v-(p-1)R$$
in view of \eqref{eqv}.

For $a=1$ and $c=1-b$ as discussed above, Proposition \ref{evolutionF} can be further simplified as follows:

\begin{proposition} Suppose $u$ is a smooth positive solution to \eqref{eq2} and $v=\frac{p}{p-1}u^{p-1}$. Then, for any constant $b$, the quantity
$$F=\frac{|\nabla v|^2}{v}-b\frac{v_t}{v}+(1-b)\frac{R}{v}$$
satisfies the equation
\begin{equation}\label{eq4}
\begin{aligned}
\mathcal{L}(F)&=2p\nabla_i F\nabla_iv-[\frac{b-1}{v}+p-1](\frac{\partial R}{\partial t}-2\nabla_i R\nabla_i v+2R_{ij}\nabla_i v \nabla_j v)\\
&\quad\ -2(p-1)|\nabla^2 v+\frac{b}{2}Rc|^2+\frac{(b-2)^2}{2}(p-1)|Rc|^2-\frac{1}{b}F^2\\
&\quad\ -[(p-1)R+\frac{2(b-1)}{b}\frac{R}{v}]F -\frac{b-1}{b}y^2-\frac{(b-1)(b-2)}{b} y \frac{R}{v}.
\end{aligned}
\end{equation}
\end{proposition}

\begin{proof} It suffices to observe that
\begin{align*}
&-b(p-1)^2(\Delta v)^2 + 2(1-b)(p-1)\frac{|\nabla v|^2}{v}\Delta v - ab(p-1)^2R\Delta v\\
\quad\ & + a(1-b)(p-1)\frac{|\nabla v|^2}{v}R +(1-b)\frac{|\nabla v|^4}{v^2}+c\frac{|\nabla v|^2}{v^2}R - ac(p-1)\frac{R^2}{v}\\
=& (p-1)\Delta v\cdot F+(1-b)(p-1)\frac{|\nabla v|^2}{v}\Delta v-c(p-1)\frac{R}{v}\Delta v + a(1-b)(p-1)\frac{|\nabla v|^2}{v}R\\
\quad\ &+(1-b)\frac{|\nabla v|^4}{v^2}+c\frac{|\nabla v|^2}{v^2}R -ac(p-1)\frac{R^2}{v}\\
=& (p-1)\Delta v\cdot F - \frac{1-b}{b}\frac{|\nabla v|^2}{v}F-c(p-1)\frac{R}{v}\Delta v+ [1-b+\frac{(1-b)^2}{b}]\frac{|\nabla v|^4}{v^2}\\
\quad\ & + [c+\frac{c(1-b)}{b}]\frac{|\nabla v|^2}{v^2}R -ac(p-1)\frac{R^2}{v}\\
=& [(p-1)\Delta v-\frac{1-b}{b}\frac{|\nabla v|^2}{v}]F - \frac{c}{b}\frac{R}{v}F-2c(p-1)\frac{R}{v}\Delta v + [c + \frac{2(1-b)c}{b}]\frac{|\nabla v|^2}{v^2}R\\
\quad\ & - 2ac(p-1)\frac{R^2}{v}+ \frac{c^2}{b}\frac{R^2}{v^2} + \frac{1-b}{b}\frac{|\nabla v|^4}{v^2}\\
=& -\frac{1}{b}[F+ab(p-1)R]F + \frac{2c}{b}\frac{R}{v}F-\frac{c^2}{b}\frac{R^2}{v^2}+c\frac{|\nabla v|^2}{v^2}R+ \frac{1-b}{b}\frac{|\nabla v|^4}{v^2}\\
=& -\frac{1}{b}F^2-[a(p-1)R-\frac{2c}{b}\frac{R}{v}]F -\frac{c^2}{b}\frac{R^2}{v^2}+c\frac{|\nabla v|^2}{v^2}R + \frac{1-b}{b}\frac{|\nabla v|^4}{v^2}.
\end{align*}
Now \eqref{eq4} follows easily by plugging in $a=1$ and $c=1-b$.
\end{proof}

To conclude this section, we recall the following important trace Li-Yau-Hamilton estimate of Hamilton \cite{Ham1993jdg} for the Ricci flow which will be crucial in the proofs of our Aronson-B\'enilan and Li-Yau-Hamilton estimates in Section 3.

\begin{theorem}[\textbf{Hamilton \cite{Ham1993jdg}}]\label{matrix harnack}
Let $(M^n, g_{ij}(t))$, $t\in[0, T)$, be a complete solution to the Ricci flow with bounded and nonnegative curvature operator, then for any 1-form $V_i$ on $M^n$, we have
$$\frac{\partial R}{\partial t} +2\nabla_i R V_i + 2R_{ij}V_iV_j + \frac{R}{t}\geq0$$
for $t\in(0,T)$.
\end{theorem}

\bigskip

\section{Aronson-B\'enilan and Li-Yau-Hamilton Estimates for PME}

In this section, we prove Aronson-B\'enilan and Li-Yau-Hamilton estimates for smooth positive solutions to PME \eqref{eq2} ($p>1$)
under the Ricci flow \eqref{rf} on complete manifold $M^n$ with bounded and nonnegative curvature operator so that Theorem \ref{matrix harnack} applies.

Since $M^n$ may be noncompact, in order to apply the maximum principle argument, we first show a local differential Harnack estimate for $v=\frac{p}{p-1}u^{p-1}$ in some parabolic cylinder $\coprod_{t\in[0, T]}B_t(O, R_0)\times\{t\}$ for any fixed point $O \in M^n$ and any positive constant $R_0>0$, where $B_t(O, R_0)$ denotes the geodesic ball centered at $O$ with radius $R_0$ under the metric $g_{ij}(t)$. Letting the radius $R_0\rightarrow \infty$ leads to a rough global differential Harnack estimate, which gives certain bound on the quantity that we are considering. This rough global differential Harnack estimate in turn enables us to adopt a method of Hamilton in \cite{Ham1993jdg} to derive the desired estimate. It turns out necessary for us to treat the cases of $b\geq2$ and $b<2$ separately.

\subsection{Case 1: $b\geq2$}

In this case, we first need the following local differential Harnack estimate.

\begin{proposition}\label{local harnack p>1 b>2}
Let $(M^n, g_{ij}(t))$, $t\in[0, T]$, be a complete solution to the Ricci flow with bounded and nonnegative curvature operator. Suppose $u$ is a smooth positive solution to \eqref{eq2} with $p>1$ and $v=\frac{p}{p-1}u^{p-1}$.  Then,  for any point $O\in M$ and any constants $R_0>0$ and $b\in[2,\infty)$, we have
$$\frac{|\nabla v|^2}{v}-b\frac{v_t}{v}-(b-1)\frac{R}{v}-\frac{d}{t}\leq b\alpha\left[ \frac{C_1v_{max}}{R_0^2}+C_2R_{max}\right]+R_{max}(b-2)\sqrt{\frac{b(p-1)\alpha}{2}}$$
on $\coprod_{t\in(0, T]}B_t(O, R_0)\times\{t\}$, where $\alpha=\frac{bn(p-1)}{2+bn(p-1)}$, $d=\max\{b\alpha, \frac{b}{2}\}$, $C_1=(8n+128+\frac{16nb^2p^2}{b-1})(p-1)$, $C_2= 16+\frac{2(b-2)}{bn}$, $v_{max}=\max_{\coprod_{t\in[0, T]}B_t(O, 2R_0)\times\{t\}}v$ and $R_{max}=\sup_{M\times[0,T]}R$.
\end{proposition}

\begin{proof}
Choose a smooth cut-off function $\eta(s)$ defined for $s\geq0$ such that
$\eta(s)=1$ for $0\leq s \leq \frac{1}{2}$, $\eta(s)=0$ for $s\geq 1$, $\eta(s)>0$ for $\frac{1}{2}<s<1$, $-16\eta^{\frac{1}{2}}\leq \eta^{\prime}\leq 0$ and $\eta^{\prime\prime}\geq -16\eta\geq -16$. For any fixed point $O\in M$ and any positive number $R_0>0$, we define
\begin{equation}\label{eqphi}
\phi(x,t)=\eta\left(\frac{r(x,t)}{2R_0}\right)
\end{equation}
on $\coprod_{t\in[0, T]} B_{t}(O, 2R_0)\times\{t\}$, where $r(x,t)$ is the distance function from $O$ at time $t$.

Set $H=t\phi F -d$ for some constant $d\geq b/2$ to be chosen later, $\tilde{y}=t\phi y$, and $\tilde{z}=t\phi z$. Then, by \eqref{f=y-bz} we have  $$t\phi F=\tilde{y}-b\tilde{z},$$
and it follows from Proposition 2.2 
that
\begin{align*}
&t\phi\mathcal{L}(H)\\
=&t\phi^2F+t^2\phi F\phi_t-(p-1)t^2v\phi F\Delta\phi - 2(p-1)t^2v\phi\nabla_i\phi\nabla_i F + t^2\phi^2\mathcal{L}(F)\\
= &\phi(\tilde{y}-b\tilde{z})+t\phi_t(\tilde{y}-b\tilde{z})-(p-1)tv\Delta\phi(\tilde{y}-b\tilde{z}) - 2(p-1)t^2v\phi\nabla_i\phi\nabla_i F\\
\quad\ &+
2p t^2\phi^2\nabla_i F\nabla_iv-t\phi^2[\frac{b-1}{v}+p-1]Q-(d-\phi)(p-1)t\phi R-(b-1)[\frac{2d}{b}-\phi]t\phi \frac{R}{v} \\
\quad\ &-2(p-1)t^2\phi^2|\nabla^2 v+\frac{b}{2}Rc|^2+\frac{(b-2)^2}{2}(p-1)t^2\phi^2|Rc|^2-\frac{1}{b}(\tilde{y}-b\tilde{z})^2\\
\quad\ &-[(p-1)R+\frac{2(b-1)}{b}\frac{R}{v}]t\phi H -\frac{b-1}{b}\tilde{y}^2-\frac{(b-1)(b-2)}{b}\cdot t\phi\frac{R}{v}\tilde{y},
\end{align*}
where we denote by \begin{equation}\label{eqQ}Q=t\frac{\partial R}{\partial t}+R-2t\nabla_i R\nabla_i v+2tR_{ij}\nabla_i v \nabla_j v.\end{equation}

From the definition of $\phi$, it is easy to see that
$$|\nabla \phi|=\frac{1} {2R_0} |\eta^{\prime}| |\nabla r(x,t)|\leq \frac{8}{R_0}\phi^{\frac{1}{2}},$$
and
\begin{align*}
\Delta\phi&= \frac{\eta^{\prime\prime}}{4R_0^2}+\eta^{\prime}\frac{\Delta r}{2R_0}\geq - \frac{4}{R_0^2}-\frac{8(n-1)}{R^2_0}=-\frac{8n}{R^2_0},
\end{align*}
where we have used the Laplacian Comparison Theorem for $\Delta r$ (see e.g. \cite{ScYa1994}) in the above. On the other hand, since along the Ricci flow we have (see e.g. Lemma B.40 in \cite{CLN2006} or Lemma 2.3.3 in \cite{CaZh2006})
$$\frac{\partial r}{\partial t}=-\inf_{\gamma}\int_{\gamma}Rc(\dot{\gamma}(s),\dot{\gamma}(s))ds\geq -R_{max}r(x, t), $$
with the infimum taking over all the shortest geodesics connecting $O$ and $x$ at time $t$, it implies that
\begin{align*}
\frac{\partial \phi}{\partial t} &= \frac{\eta^{\prime}}{2R_0}\frac{\partial r}{\partial t}\leq 16R_{max}.
\end{align*}


If $H\leq 0$ in $\coprod_{t\in[0, T]}B_t(O, 2R_0)\times\{t\}$, then clearly the estimate that we seek is automatically true. Otherwise, since by definition $H\leq 0$ on the parabolic boundary of $\coprod_{t\in[0, T]}B_t(O, 2R_0)\times\{t\}$, we may assume that $H$ achieves its positive maximum on the parabolic cylinder $\coprod_{t\in[0, T]}B_t(O, 2R_0)\times\{t\}$ at some time $t_0>0$ and some interior point $x_0$. Then, at $(x_0, t_0)$, one has
$$\tilde{y}-b\tilde{z}=H(x_0,t_0)+d>0,\ F\nabla\phi=-\phi\nabla F, \  \textrm{and}\ \mathcal{L}(H)(x_0,t_0)\geq0.$$
Moreover, since
$$2p t_0^2\phi^2\nabla_i F\nabla_iv=- 2p t_0^2\phi F \nabla_i\phi\nabla_i v\leq 2p t_0^2\phi F|\nabla\phi||\nabla v|\leq \frac{16p}{R_0}\tilde{y}^{\frac{1}{2}}(t_0v)^{\frac{1}{2}}(\tilde{y}-b\tilde{z}),$$
by Theorem \ref{matrix harnack}, one has
\begin{align*}
0\leq &t_0\phi\mathcal{L}(H)\\
\leq & (\tilde{y}-b\tilde{z})+ 16t_0R_{max}(\tilde{y}-b\tilde{z}) + (8n+128)(p-1)\frac{t_0v}{R_0^2}(\tilde{y}-b\tilde{z}) +
\frac{16p}{R_0}\tilde{y}^{\frac{1}{2}}(t_0v)^{\frac{1}{2}}(\tilde{y}-b\tilde{z})\\
\quad\ &-\frac{2}{b^2n(p-1)}[\tilde{y}-b\tilde{z}+(b-1)\tilde{y}-\frac{b(b-2)}{2}(p-1)t_0\phi R]^2+\frac{(b-2)^2}{2}(p-1)t_0^2\phi^2R^2\\
\quad\ & -\frac{1}{b}(\tilde{y}-b\tilde{z})^2\\
\end{align*}
\begin{align*}
\leq & (\tilde{y}-b\tilde{z})+ 16t_0R_{max}(\tilde{y}-b\tilde{z}) + (8n+128)(p-1)\frac{t_0v}{R_0^2}(\tilde{y}-b\tilde{z}) +
\frac{16p}{R_0}\tilde{y}^{\frac{1}{2}}(t_0v)^{\frac{1}{2}}(\tilde{y}-b\tilde{z})\\
\quad\ &-\frac{1}{b\alpha}(\tilde{y}-b\tilde{z})^2- \frac{4(b-1)}{b^2n(p-1)}\tilde{y}(\tilde{y}-b\tilde{z})+\frac{2(b-2)}{bn}t_0\phi R(\tilde{y}-b\tilde{z})+\frac{(b-2)^2(p-1)}{2}t_0^2\phi^2R^2\\
=& (\tilde{y}-b\tilde{z})\left[-\frac{4(b-1)}{b^2n(p-1)}\tilde{y}+\frac{16p}{R_0}\tilde{y}^{\frac{1}{2}}(t_0v_{max})^{\frac{1}{2}}+(8n+128)(p-1)\frac{t_0v_{max}}{R_0^2}\right]\\
\quad\ & + [t_0(16+\frac{2(b-2)}{bn})R_{max}+1](\tilde{y}-b\tilde{z})+\frac{(b-2)^2}{2}(p-1)t_0^2R_{max}^2-\frac{1}{b\alpha}(\tilde{y}-b\tilde{z})^2 \\
\leq &  (\tilde{y}-b\tilde{z})\left[ \frac{t_0v_{max}}{R_0^2}(8n+128+\frac{16nb^2p^2}{b-1})(p-1)+(16+\frac{2(b-2)}{bn})t_0R_{max}+1\right]\\
\quad\ & +\frac{(b-2)^2}{2}(p-1)t_0^2R_{max}^2 -\frac{1}{b\alpha}(\tilde{y}-b\tilde{z})^2.
\end{align*}

Thus, at $(x_0, t_0)$,
\begin{align*}
\tilde{y}-b\tilde{z}&\leq b\alpha\left[ \frac{t_0v_{max}}{R_0^2}(8n+128+\frac{16nb^2p^2}{b-1})(p-1)+(16+\frac{2(b-2)}{bn})t_0R_{max}+1\right]\\
&\quad\ +t_0R_{max}(b-2)\sqrt{\frac{b(p-1)\alpha}{2}},
\end{align*}
i.e.,
\begin{align*}
H&\leq b\alpha\left[ \frac{t_0v_{max}}{R_0^2}(8n+128+\frac{16nb^2p^2}{b-1})(p-1)+(16+\frac{2(b-2)}{bn})t_0R_{max}+1\right]\\
&\quad\ +t_0R_{max}(b-2)\sqrt{\frac{b(p-1)\alpha}{2}}-d.
\end{align*}

Therefore,  for $x\in \coprod_{t\in[0, T]} B_t(x_0, R_0)\times\{t\}$, we have
\begin{align*}
tF&\leq H(x_0,t_0)+d\\
&\leq b\alpha\left[ \frac{tv_{max}}{R_0^2}(8n+128+\frac{16nb^2p^2}{b-1})(p-1)+(16+\frac{2(b-2)}{bn})tR_{max}+1\right]\\
&\quad\ +tR_{max}(b-2)\sqrt{\frac{b(p-1)\alpha}{2}}.
\end{align*}
This finishes the proof.
\end{proof}

If $u$ is bounded on $M\times[0,T]$, then letting $R_0\rightarrow\infty$, we immediately get

\begin{corollary}\label{noncompact harnack p>1 b>2}
Let $(M^n, g_{ij}(t))$, $t\in[0, T]$, be a complete solution to the Ricci flow with bounded and nonnegative curvature operator. Suppose $u$ is a bounded smooth positive solution to \eqref{eq2} with $p>1$. Then, for $v=\frac{p}{p-1}u^{p-1}$ and $b\in[2,\infty)$, we have
$$\frac{|\nabla v|^2}{v}-b\frac{v_t}{v}-(b-1)\frac{R}{v}-\frac{d}{t}\leq \left[C_2b\alpha+(b-2)\sqrt{\frac{b(p-1)\alpha}{2}}\right]R_{max}$$
on $M\times(0,T]$, where $\alpha$, $d$, and $R_{max}$ are the constants in Proposition \ref{local harnack p>1 b>2}.
\end{corollary}

To further refine the differential Harnack inequality above, we follow a method of Hamilton which uses the following distance-like function (see \cite{Ham1993jdg}).

\begin{lemma}\label{distance function}
Let $g_{ij}(t)$, $t\in[0, T]$, be a complete solution to the Ricci flow on $M^n$ with bounded curvature tensor. Then, there exists a smooth function $f(x)$ on $M$  and a positive constant $C>0$ such that $f\geq1$,
$f(x)\rightarrow\infty$ as $d_0(x,O)\rightarrow\infty$ (for some fixed point $O\in M$),
$$|\nabla f|_{g(t)}\leq C,\qquad and \qquad |\nabla\nabla f|_{g(t)}\leq C$$
on $M\times[0,T]$.
\end{lemma}

\begin{theorem}\label{refined noncompact harnack p>1 b>2}
Let $(M^n, g_{ij}(t))$, $t\in[0, T]$, be a complete solution to the Ricci flow with bounded and nonnegative curvature operator. If $u$ is a bounded smooth positive solution to \eqref{eq2} with $p>1$, then for $v=\frac{p}{p-1}u^{p-1}$ and $b\in[2,\infty)$, we have
$$\frac{|\nabla v|^2}{v}-b\frac{v_t}{v}-(b-1)\frac{R}{v}-\frac{d}{t}\leq C_0(b-2)R_{max}$$
on $M\times(0,T]$, where $\alpha=\frac{bn(p-1)}{2+bn(p-1)}$, $d=\max\{b\alpha, \frac{b}{2}\}$, $C_0=\left[\frac{2\alpha}{n}+\sqrt{\frac{b\alpha(p-1)}{2}}\right]$, and $R_{max}=\sup_{M\times[0,T]}R$.
\end{theorem}

\begin{proof}
Let $H=t(F-K)-d$ for $d=\max\{b\alpha, \frac{b}{2}\}$ and some constant $K>0$ to be determined. Then, from \eqref{eq4}, we have
\begin{align*}
t\mathcal{L}(H)&= t(F-K)+t^2\mathcal{L}(F)\\
&= H+d +2p t\nabla_i H\nabla_iv-[\frac{b-1}{v}+p-1]tQ+tR[\frac{b-1}{v}+p-1]\\
&\quad\ -2(p-1)t^2|\nabla^2 v+\frac{b}{2}Rc|^2+\frac{(b-2)^2}{2}(p-1)t^2|Rc|^2-\frac{1}{b}(H+tK+d)^2\\
&\quad\ -t[(p-1)R+\frac{2(b-1)}{b}\frac{R}{v}](H+tK+d) -\frac{b-1}{b}t^2y^2-\frac{(b-1)(b-2)}{b}\cdot \frac{R}{v}t^2y,
\end{align*}
where $Q$ is the trace Li-Yau-Hamilton quantity for the evolving scalar curvature defined in \eqref{eqQ}.

By Theorem \ref{matrix harnack}, we have
\begin{align*}
t\mathcal{L}(H)
&\leq 2p t\nabla H \cdot \nabla v-(d-1)(p-1)tR - (b-1)(\frac{2d}{b}-1)\frac{tR}{v}+\frac{(b-2)^2}{2}(p-1)t^2R_{max}^2\\
&\quad\ -\frac{2}{b^2n(p-1)}[H+tK+d+(b-1)ty-\frac{b(b-2)}{2}(p-1)tR]^2\\
&\quad\ -\frac{1}{b}(H+tK+d)^2-t[(p-1)R+\frac{2(b-1)}{b}\frac{R}{v}]H -\frac{b-1}{b}t^2y^2+H+d\\
&\leq 2p t\nabla H \cdot \nabla v-\frac{1}{b\alpha}(H+tK+d)^2-\frac{4(b-1)}{b^2n(p-1)}ty(H+tK+d)\\
&\quad\  +\frac{2(b-2)}{bn}tR(H+tK+d)+\frac{(b-2)^2}{2}(p-1)t^2R_{max}^2\\
&\quad\ -t[(p-1)R+\frac{2(b-1)}{b}\frac{R}{v}]H-\frac{b-1}{b}t^2y^2+H+d.
\end{align*}

Set $\tilde{H}=H-\epsilon \psi$, where $\psi=e^{At}f$ for some constant $A$ to be determined, and $f$ is the function in Lemma \ref{distance function}. Then,
\begin{align*}
&t\mathcal{L}(\tilde{H})\\
=&t\mathcal{L}(H) -\epsilon At\psi + (p-1)\epsilon tve^{At}(\Delta f)\\
\leq &2p t\nabla \tilde{H}\cdot\nabla v+2p t \epsilon e^{At}\nabla f\cdot\nabla v  -\frac{1}{b\alpha}(H+tK+d)^2-\frac{4(b-1)}{b^2n(p-1)}ty(H+tK+d)\\
\quad\ &+\frac{2(b-2)}{bn}tR(H+tK+d) +\frac{(b-2)^2}{2}(p-1)t^2R_{max}^2\\
\quad\ &-t[(p-1)R+\frac{2(b-1)}{b}\frac{R}{v}]\tilde{H}-\frac{b-1}{b}t^2y^2+H+d-\epsilon At\psi + (p-1)\epsilon tve^{At}(\Delta f).
\end{align*}

From Corollary \ref{noncompact harnack p>1 b>2}, we know that $\tilde{H}<0$ at $t=0$ and also outside a fixed compact subset of $M$ for all $t\in(0,T]$. We claim that $\tilde{H}<0$ on $M^n\times [0,T]$. If not, then $\tilde{H}=0$ at some first time $t=t_0>0$ and at some point $x_0\in M$. Then, at $(x_0, t_0)$, we have
$$H=\epsilon\psi,\qquad \nabla\tilde{H}=0, \qquad \textrm{and}\qquad  \mathcal{L}(\tilde{H})\geq 0.$$
Thus,
\begin{align*}
0&\leq t_0\mathcal{L}(\tilde{H})\\
&< Ct_0\epsilon\psi|\nabla v| - \frac{1}{b\alpha}(\epsilon\psi)^2-(\frac{2d}{b\alpha}-1)\epsilon\psi-(\frac{d^2}{b\alpha}-d)-\frac{b-1}{b}\frac{t_0^2|\nabla v|^4}{v^2}- At_0\epsilon\psi\\
&\quad\ + \left[Cv+\frac{2(b-2)}{bn}R_{max}\right]t_0\epsilon\psi-\frac{1}{b\alpha}t_0^2K^2+\frac{(b-2)^2}{2}(p-1)t_0^2R_{max}^2\\
&\quad\ +\frac{2(b-2)}{bn}t_0^2KR_{max}-\frac{2d}{b\alpha}t_0K+\frac{2(b-2)d}{bn}t_0R_{max}.
\end{align*}
Now if we choose $K>0$ large enough (e.g., $K=[\frac{2\alpha}{n}+\sqrt{\frac{b\alpha(p-1)}{2}}](b-2)R_{max}$) so that
$$-\frac{1}{b\alpha}K^2+\frac{(b-2)^2}{2}(p-1)R_{max}^2+\frac{2(b-2)}{bn}KR_{max}\leq 0$$
and
$$-\frac{2d}{b\alpha}K+\frac{2(b-2)d}{bn}R_{max}\leq0,$$
then we have
\begin{align*}
0&\leq t_0\mathcal{L}(\tilde{H})\\
&\leq - \frac{1}{b\alpha}(\epsilon\psi)^2-\frac{b-1}{b}\frac{t_0^2|\nabla v|^4}{v^2}+ \frac{2\sqrt{b-1}}{b\sqrt{\alpha}}\frac{t_0|\nabla v|^2}{v}\epsilon\psi\\
&\quad\ +\left[\frac{b}{8\sqrt{b-1}}Cv+\frac{2(b-2)}{bn}R_{max}\right]t_0\epsilon\psi-\epsilon At_0\psi.
\end{align*}
It is easy to see that this would lead to a contradiction by furthermore choosing $ A>\frac{b}{8\sqrt{b-1}}Cv_{max}+\frac{2(b-2)}{bn}R_{max}$.

Therefore, $\tilde{H}=H-\epsilon e^{At}f<0$ on $M\times [0,T]$. Letting $\epsilon\rightarrow 0$, it follows that $$H\leq 0.$$
\end{proof}

Now, taking $b=2$ in Theorem \ref{refined noncompact harnack p>1 b>2} we immediately get Theorem \ref{theorem 1}.

\subsection{Case 2: $1\leq b\leq 2$}

In this case, since $2-b\geq0$, we can actually get a simpler local differential Harnack inequality than the one in Proposition \ref{local harnack p>1 b>2}.

\begin{proposition}\label{local harnack p>1 b<2}
Let $(M^n, g_{ij}(t))$, $t\in[0, T]$, be a complete solution to the Ricci flow with bounded and nonnegative curvature operator. If $u$ is a smooth positive solution to \eqref{eq2} with $p>1$, then for $v=\frac{p}{p-1}u^{p-1}$, any point $O\in M$, any constants $R_0>0$ and $b\in(1,2]$, we have
\begin{equation*}
\frac{|\nabla v|^2}{v}-b\frac{v_t}{v}-(b-1)\frac{R}{v}-\frac{d}{t}\leq b\alpha\left[ \frac{C_1v_{max}}{R_0^2}+8R_{max}\right]+R_{max}(2-b)\sqrt{\frac{b(p-1)\alpha}{2}}
\end{equation*}
on $\coprod_{t\in(0, T]}B_t(O, R_0)\times\{t\}$, where $d=\max\{b\alpha, 1\}$, $\alpha=\frac{bn(p-1)}{2+bn(p-1)}$, $C_1=(8n+128+\frac{16nb^2p^2}{b-1})(p-1)$, $v_{max}=\max_{\coprod_{t\in[0, T]}B_t(O, 2R_0)\times\{t\}}v$ and $R_{max}=\sup_{M\times[0,T]}R$.
\end{proposition}

\begin{proof}
Let $\phi(x,t)$ be the cut-off function defined in \eqref{eqphi}. Denote by $H=t\phi F -d$ for $d\geq 1$, $\tilde{y}=t\phi y$ and $\tilde{z}=t\phi z$. Then
$$t\phi F=\tilde{y}-b\tilde{z},$$
and Proposition 2.2 implies that
\begin{align*}
&t\phi\mathcal{L}(H)\\
=&t\phi\mathcal{L}(t\phi F)\\
=& \phi(\tilde{y}-b\tilde{z})+t\phi_t(\tilde{y}-b\tilde{z})-(p-1)t^2v\phi F\Delta\phi - 2(p-1)t^2v\phi\nabla_i\phi\nabla_i F +
2p t^2\phi^2\nabla_i F\nabla_iv\\
\quad\ &-t\phi^2[\frac{b-1}{v}+p-1]Q +t\phi^2R[\frac{b-1}{v}+p-1]-2(p-1)t^2\phi^2|\nabla^2 v+Rc-\frac{2-b}{2}Rc|^2\\
\quad\ &+\frac{(b-2)^2}{2}(p-1)t^2\phi^2|Rc|^2-\frac{1}{b}(\tilde{y}-b\tilde{z})^2-[(p-1)R+\frac{2(b-1)}{b}\frac{R}{v}]t\phi(H+d)\\
\quad\ & -\frac{b-1}{b}\tilde{y}^2+\frac{(b-1)(2-b)}{b}\cdot \frac{R}{v}t\phi\tilde{y},
\end{align*}
where $Q\geq0$ is the quantity in \eqref{eqQ}.

Recall that
$$|\frac{\partial \phi}{\partial t}|\leq 16R_{max},\qquad \Delta\phi\geq-\frac{8n}{R^2_0},\ \textrm{and} \qquad |\nabla \phi|\leq \frac{8}{R_0}\phi^{\frac{1}{2}}.$$

If $H\leq 0$ in $\coprod_{t\in[0, T]}B_t(O, 2R_0)\times\{t\}$, then we are done. Otherwise, since $H\leq 0$ on the parabolic boundary of $\coprod_{t\in[0, T]}B_t(O, 2R_0)\times\{t\}$, we may assume that $H$ achieves a positive maximum at time $t_0>0$ and some interior point $x_0$. Thus, at $(x_0, t_0)$, we have
$$\tilde{y}-b\tilde{z}=H(x_0,t_0)+d>0,\quad F\nabla\phi=-\phi\nabla F,\quad \textrm{and}\quad \mathcal{L}(H)(x_0,t_0)\geq0.$$
Thus, one has
\begin{align*}
0\leq &t_0\phi\mathcal{L}(H)\\
\leq& (\tilde{y}-b\tilde{z})+ 16t_0R_{max}(\tilde{y}-b\tilde{z}) + (8n+128)(p-1)\frac{t_0v}{R_0^2}(\tilde{y}-b\tilde{z}) +
\frac{16p}{R_0}\tilde{y}^{\frac{1}{2}}(t_0v)^{\frac{1}{2}}(\tilde{y}-b\tilde{z})\\
\quad\ &-\frac{2}{b^2n(p-1)}[\tilde{y}-b\tilde{z}+(b-1)\tilde{y}]^2+\frac{(b-2)^2}{2}(p-1)t_0^2\phi^2R^2-\frac{1}{b}(\tilde{y}-b\tilde{z})^2\\
\leq & (\tilde{y}-b\tilde{z})+ 16t_0R_{max}(\tilde{y}-b\tilde{z}) + (8n+128)(p-1)\frac{t_0v}{R_0^2}(\tilde{y}-b\tilde{z}) +
\frac{16p}{R_0}\tilde{y}^{\frac{1}{2}}(t_0v)^{\frac{1}{2}}(\tilde{y}-b\tilde{z})\\
\quad\ &-\frac{1}{b\alpha}(\tilde{y}-b\tilde{z})^2- \frac{4(b-1)}{b^2n(p-1)}\tilde{y}(\tilde{y}-b\tilde{z})+\frac{(b-2)^2}{2}(p-1)t_0^2\phi^2R_{max}^2\\
\leq& -\frac{1}{b\alpha}(\tilde{y}-b\tilde{z})^2 + (\tilde{y}-b\tilde{z})\left[-\frac{4(b-1)}{b^2n(p-1)}\tilde{y}+\frac{16p}{R_0}\tilde{y}^{\frac{1}{2}}(t_0v_{max})^{\frac{1}{2}}+(8n+128)(p-1)\frac{t_0v_{max}}{R_0^2}\right]\\
\quad\ & + [16t_0R_{max}+1](\tilde{y}-b\tilde{z})+\frac{(b-2)^2}{2}(p-1)t_0^2R_{max}^2\\
\leq & -\frac{1}{b\alpha}(\tilde{y}-b\tilde{z})^2 + (\tilde{y}-b\tilde{z})\left[ \frac{t_0v_{max}}{R_0^2}(8n+128+\frac{16nb^2p^2}{b-1})(p-1)+8t_0R_{max}+1\right]\\
\quad\ & +\frac{(b-2)^2}{2}(p-1)t_0^2R_{max}^2.
\end{align*}

Therefore, at $(x_0, t_0)$, we have
$$\tilde{y}-b\tilde{z}\leq b\alpha\left[ C_1\frac{t_0v_{max}}{R_0^2}+16t_0R_{max}+1\right]+t_0R_{max}(2-b)\sqrt{\frac{b(p-1)\alpha}{2}},$$
where $C_1=(8n+128+\frac{16nb^2p^2}{b-1})(p-1)$.
i.e.,
$$H\leq b\alpha\left[C_1\frac{t_0v_{max}}{R_0^2}+16t_0R_{max}+1\right]+t_0R_{max}(2-b)\sqrt{\frac{b(p-1)\alpha}{2}}-d.$$

The Proposition follows immediately since $\phi=1$ for $x\in \coprod_{t\in[0, T]} B_t(x_0, R_0)\times\{t\}$.
\end{proof}

For bounded $u$, letting $R_0\rightarrow \infty$ in the above Proposition, one gets

\begin{corollary}\label{noncompact harnack p>1 b<2}
Let $(M^n,g_{ij}(t))$, $t\in[0, T]$, be a complete solution to the Ricci flow with bounded and nonnegative curvature operator. If $u$ is a bounded smooth positive solution to \eqref{eq2} with $p>1$, then for $v=\frac{p}{p-1}u^{p-1}$ and $b\in(1,2]$, we have
$$\frac{|\nabla v|^2}{v}-b\frac{v_t}{v}-(b-1)\frac{R}{v}-\frac{d}{t}\leq \left[16b\alpha+(2-b)\sqrt{\frac{b(p-1)\alpha}{2}}\right]R_{max}$$
on $M\times(0,T]$, where $\alpha=\frac{bn(p-1)}{2+bn(p-1)}$, $d=\max\{b\alpha, 1\}$, and $R_{max}=\sup_{M\times[0,T]}R$.
\end{corollary}

Again,  we can refine the above differential Harnack estimate as follows.

\begin{theorem}\label{refined noncompact harnack p>1 b<2}
Let $(M^n, g_{ij}(t))$, $t\in[0, T]$, be a complete solution to the Ricci flow with bounded and nonnegative curvature operator. If $u$ is a bounded smooth positive solution to \eqref{eq2} with $p>1$, then for $v=\frac{p}{p-1}u^{p-1}$ and $b\in(1,2]$, we have
$$\frac{|\nabla v|^2}{v}-b\frac{v_t}{v}-(b-1)\frac{R}{v}-\frac{d}{t}\leq C_0(2-b)R_{max}$$
on $M\times(0,T]$, where $\alpha=\frac{bn(p-1)}{2+bn(p-1)}$, $d=\max\{b\alpha, \frac{b}{2}\}$, $C_0=\sqrt{\frac{b\alpha(p-1)(n-1)}{2n}}$, and $R_{max}=\sup_{M\times[0,T]}R$.
\end{theorem}

\begin{proof}

Let $H=t(F-K)-d$ for $d=\max\{b\alpha, \frac{b}{2}\}$ and some constant $K>0$ to be determined. From \eqref{eq4}, we have
\begin{align*}
t\mathcal{L}(H)&= t(F-K)+t^2\mathcal{L}(F)\\
&= H+d+2p t\nabla_i H\nabla_iv-[\frac{b-1}{v}+p-1]tQ+tR[\frac{b-1}{v}+p-1]\\
&\quad\ -2(p-1)t^2|\nabla^2 v+\frac{b}{2}Rc|^2+\frac{(b-2)^2}{2}(p-1)t^2|Rc|^2-\frac{1}{b}(H+tK+d)^2\\
&\quad\ -t[(p-1)R+\frac{2(b-1)}{b}\frac{R}{v}](H+tK+d) -\frac{b-1}{b}t^2y^2+\frac{(b-1)(2-b)}{b}\cdot \frac{R}{v}t^2y,
\end{align*}
where $Q$ is the quantity in \eqref{eqQ}.

Set $\tilde{H}=H-\epsilon \psi$, where $\psi=e^{At}f$ for some constant $A$ to be determined, and $f$ is the function in Lemma \ref{distance function}. It then follows from Theorem \ref{matrix harnack} and the choice of $d$ that
\begin{align*}
t\mathcal{L}(\tilde{H}) &\leq 2p t\nabla_i \tilde{H}\nabla_iv+2p t \epsilon e^{At}\nabla_i f\nabla_i v  -(d-1)(p-1)tR-\frac{1}{b\alpha}(H+tK+d)^2\\
&\quad\ -\frac{4(b-1)}{b^2n(p-1)}ty(H+tK+d) -\frac{2(2-b)}{bn}tR(H+tK+d)\\
&\quad\ +\frac{(b-2)^2(n-1)}{2n}(p-1)t^2R_{max}^2-t[(p-1)R+\frac{2(b-1)}{b}\frac{R}{v}]\tilde{H}-\frac{(b-1)^2}{b\alpha}t^2y^2\\
&\quad\ +H+d-\epsilon At\psi + (p-1)\epsilon tve^{At}(\Delta f).
\end{align*}

From Corollary \ref{noncompact harnack p>1 b<2}, we know that $\tilde{H}<0$ at $t=0$ and outside a fixed compact subset of $M$ for all $t\in(0,T]$. Again we claim $\tilde{H}<0$ on $M\times[0,T]$. If not, then $\tilde{H}=0$ at some first time $t_0>0$ and some point $x_0\in M$. Then, at $(x_0, t_0)$, we have
$$H=\epsilon\psi>0,\qquad \nabla\tilde{H}=0,\qquad\textrm{and}\qquad \mathcal{L}(\tilde{H})\geq 0.$$
Thus,
\begin{align*}
0&\leq t_0\mathcal{L}(\tilde{H})\\
&\leq 2p t_0 \epsilon e^{At_0}\nabla_i f\nabla_i v -[d+\frac{2(2-b)}{bn(p-1)}d-1](p-1)t_0R -\frac{1}{b\alpha}(\epsilon\psi)^2-\frac{1}{b\alpha}t_0^2K^2-\frac{1}{b\alpha}d^2\\
&\quad\  -\frac{2d}{b\alpha}\epsilon\psi+\frac{(b-2)^2(n-1)}{2n}(p-1)t_0^2R_{max}^2-\frac{(b-1)^2}{b\alpha}t_0^2y^2+\epsilon\psi+d\\
&\quad\ -\epsilon At_0\psi + (p-1)\epsilon t_0ve^{At_0}(\Delta f)\\
&< Ct_0\epsilon\psi|\nabla v| - [(\frac{2-2b}{bn(p-1)}+\frac{1}{\alpha})d-1]t_0R- \frac{1}{b\alpha}(\epsilon\psi)^2-(\frac{2d}{b\alpha}-1)\epsilon\psi-(\frac{d^2}{b\alpha}-d)\\
&\quad\ -\frac{(b-1)^2}{b\alpha}\frac{t_0^2|\nabla v|^4}{v^2}-\epsilon At_0\psi + Cvt_0\epsilon\psi-\frac{1}{b\alpha}t_0^2K^2+\frac{(b-2)^2(n-1)}{2n}(p-1)t_0^2R_{max}^2.
\end{align*}

It is not hard to check that with our choices of $b$ and $d$,
$$(\frac{2-2b}{bn(p-1)}+\frac{1}{\alpha})d-1\geq 0.$$

By setting $K=\sqrt{\frac{b\alpha(p-1)(n-1)}{2n}}(2-b)R_{max}$, we have
\begin{align*}
0&\leq t_0\mathcal{L}(\tilde{H})\\
&\leq - \frac{1}{b\alpha}(\epsilon\psi)^2-\frac{(b-1)^2}{b\alpha}\frac{t_0^2|\nabla v|^4}{v^2}+ \frac{2(b-1)}{b\alpha}\frac{t_0|\nabla v|^2}{v}\epsilon\psi+\frac{b}{b-1}Cvt_0\epsilon\psi-\epsilon At_0\psi.
\end{align*}
This would lead to a contradiction if we choose $ A>\frac{b}{b-1}Cv_{max}$.

Therefore we have shown $\tilde{H}=H-\epsilon e^{At}f<0$ on $M\times [0,T]$. Letting $\epsilon\rightarrow 0$, we get $$H\leq 0.$$

\end{proof}

Now, by combining Theorem \ref{refined noncompact harnack p>1 b>2} and Theorem \ref{refined noncompact harnack p>1 b<2}, we get Theorem \ref{harnack PME}.\\

Moreover, letting $b\rightarrow1$ in Theorem \ref{refined noncompact harnack p>1 b<2}, we have

\begin{theorem}\label{noncompact harnack1 p>1 b=1}
Let $(M^n, g_{ij}(t))$, $t\in[0, T]$, be a complete solution to the Ricci flow with bounded nonnegative curvature operator. If $u$ is a bounded smooth positive solution to \eqref{eq2} with $p>1$, then for $v=\frac{p}{p-1}u^{p-1}$, we have
$$\frac{|\nabla v|^2}{v}-\frac{v_t}{v}-\frac{d}{t}\leq C_0R_{max}$$
on $M\times(0,T]$, where $d=\max\{\alpha,\frac{1}{2}\}$, $\alpha=\frac{n(p-1)}{2+n(p-1)}$, $C_0=\sqrt{\frac{\alpha(p-1)(n-1)}{2n}}$, and $R_{max}=\sup_{M\times[0,T]}R$.
\end{theorem}

In case $|\nabla v|$ is also bounded, we can improve the coefficient of $\frac{1}{t}$ in Theorem \ref{noncompact harnack1 p>1 b=1}.

\begin{theorem}\label{noncompact harnack2 p>1 b=1}
Suppose that $(M^n, g_{ij}(t))$, $t\in[0, T]$, is a complete solution to the Ricci flow with bounded nonnegative curvature operator. Let $u$ be a smooth positive solution to \eqref{eq2} for $p>1$ and $v=\frac{p}{p-1}u^{p-1}$. Assume that both $v$ and $|\nabla v|$ are bounded on $M\times[0,T]$.Then, we have
\begin{equation}\label{eq noncompact harnack2 p>1 b=1}
\frac{|\nabla v|^2}{v}-\frac{v_t}{v} -\frac{d}{t}\leq C_0R_{max}
\end{equation}
for all $t\in(0, T]$, where $d=\frac{n(p-1)}{2+n(p-1)}$, $C_0=\sqrt{\frac{\alpha(p-1)(n-1)}{2n}}$, and $R_{max}=\sup_{M\times[0,T]}R$.
\end{theorem}

\begin{proof}
Let $H=t(F-K)-d$ for $d=\frac{n(p-1)}{2+n(p-1)}$, and some constant $K>0$ to be determined. From \eqref{eq4}, we have
\begin{align*}
t\mathcal{L}(H)&= 2p t\nabla_i H\nabla_iv-(p-1)tQ+(p-1)tR-2(p-1)t^2|\nabla^2 v+\frac{1}{2}Rc|^2\\
&\quad\ +\frac{1}{2}(p-1)t^2|Rc|^2-(H+tK+d)^2 -(p-1)tR(H+tK+d)+H+d,
\end{align*}
where $Q$ is the quantity in \eqref{eqQ}.

Theorem \ref{matrix harnack} implies that
\begin{align*}
t\mathcal{L}(H)
&\leq 2p t\nabla_i H\nabla_iv-(d-1)(p-1)tR-\frac{1}{d}(H+tK+d)^2-\frac{2}{n}tR(H+tK+d)\\
&\quad\ -\frac{1}{2n}(p-1)t^2R^2+\frac{1}{2}(p-1)t^2R^2-(p-1)tRH+H+d.
\end{align*}

Now, let $\tilde{H}=H-\epsilon \psi$, where $\psi=e^{At}f$ for some constant $A$ to be determined, and $f$ is the function in Lemma \ref{distance function}. Then
\begin{align*}
t\mathcal{L}(\tilde{H})&\leq 2p t\nabla_i \tilde{H}\nabla_iv+2p t \epsilon e^{At}\nabla_i f\nabla_i v  -(d-1)(p-1)tR-\frac{1}{d}(H+tK+d)^2\\
&\quad\ -\frac{2}{n}tR(H+tK+d)+\frac{(n-1)}{2n}(p-1)t^2R_{max}^2-(p-1)tR\tilde{H}+H+d\\
&\quad\ -\epsilon At\psi + (p-1)\epsilon tve^{At}(\Delta f).
\end{align*}

From Corollary \ref{noncompact harnack p>1 b<2}, we know that $\tilde{H}<0$ at $t=0$ and outside a fixed compact subset of $M$ for $t\in(0,T]$. Suppose that $\tilde{H}=0$ for the first time at $t=t_0>0$ and some point $x_0\in M$. Then, at $(x_0, t_0)$, we have
$$H=\epsilon\psi>0,\qquad \nabla\tilde{H}=0,\qquad\textrm{and}\qquad \mathcal{L}(\tilde{H})\geq 0.$$
Thus,
\begin{align*}
0&\leq t_0\mathcal{L}(\tilde{H})\\
&\leq 2p t_0 \epsilon e^{At_0}\nabla_i f\nabla_i v  -\frac{1}{d}(\epsilon\psi)^2-d-\frac{1}{d}t_0^2K^2-2\epsilon\psi+\frac{(n-1)}{2n}(p-1)t_0^2R_{max}^2\\
&\quad\ +\epsilon\psi+d-\epsilon At_0\psi + (p-1)\epsilon t_0ve^{At_0}(\Delta f)\\
&< Ct_0\epsilon\psi|\nabla v| - \frac{1}{d}(\epsilon\psi)^2-\epsilon\psi-\epsilon At_0\psi + Cvt_0\epsilon\psi-\frac{1}{d}t_0^2K^2+\frac{(n-1)}{2n}(p-1)t_0^2R_{max}^2.
\end{align*}
Now if we choose $K=\sqrt{\frac{d(p-1)(n-1)}{2n}}R_{max}$, then we get
\begin{align*}
0&\leq t_0\mathcal{L}(\tilde{H})\leq t_0\epsilon\psi (C|\nabla v|+Cv-A)<0,
\end{align*}
provided we choose furthermore $ A>C(v_{max}+|\nabla v|_{max})$. Hence, we get a contradiction.
Therefore, $\tilde{H}=H-\epsilon e^{At}f<0$ on $M\times [0,T]$. Letting $\epsilon\rightarrow 0$, we see that $H\leq 0$.

\end{proof}

Finally, by Integrating the differential Harnack quantity in Theorems \ref{harnack PME} along space-time curves, we obtain the following Harnack inequalities for $v$.

\begin{corollary}\label{cor1} Under the same assumptions as in Theorem \ref{harnack PME}, if moreover, $v_{\min}=\inf_{M\times[0,T]}v>0$, then for any points $x_1$,$ x_2\in M$, $0<t_1<t_2\leq T$ and $b\geq 1$, we have
\begin{equation}\label{eq cor1}
v(x_1, t_1)\leq v(x_2, t_2)\cdot (\frac{t_2}{t_1})^{d/b}exp\left( \frac{\Gamma}{v_{min}}+\frac{|b-2|}{b}C_0 R_{max}(t_2-t_1)\right),
\end{equation}
where $d$, $C_0$ and $R_{max}$ are the constants in Theorem \ref{harnack PME}, and $\Gamma=\inf_{\gamma}\int_{t_1}^{t_2}(\frac{b-1}{b}R+\frac{b}{4}\left|\frac{d\gamma}{d\tau}\right|^2_{g_{ij}(\tau)})d\tau$ with the infimum taking over all smooth curves $\gamma(\tau)$ in $M$, $\tau\in[t_1, t_2]$, so that $\gamma(t_1)=x_1$ and $\gamma(t_2)=x_2$.

\end{corollary}

\begin{proof}
For any curve $\gamma(\tau)$, $\tau\in[t_1, t_2]$, from $x_1$ to $x_2$, we have
\begin{align*}
\log\frac{v(x_2, t_2)}{v(x_1,t_1)}&=\int_{t_1}^{t_2}\frac{d}{d\tau}\log v(\gamma(\tau),\tau)d\tau\\
&=\int_{t_1}^{t_2}(\frac{v_{\tau}}{v}+\frac{\nabla v}{v}\cdot\frac{d\gamma}{d\tau})d\tau\\
&\geq\int_{t_1}^{t_2}(\frac{v_{\tau}}{v}-\frac{|\nabla v|^2}{bv} - \frac{b}{4v}\left|\frac{d\gamma}{d\tau}\right|^2_{g_{ij}(\tau)})d\tau
\end{align*}
Theorem \ref{refined noncompact harnack p>1 b>2} implies that
\begin{align*}
\log\frac{v(x_2, t_2)}{v(x_1,t_1)}
&\geq\int_{t_1}^{t_2}(-\frac{(b-1)R}{bv_{min}}-\frac{d}{b\tau}-\frac{|b-2|}{b}C_0 R_{max} - \frac{b}{4v_{min}}\left|\frac{d\gamma}{d\tau}\right|^2_{g_{ij}(\tau)})d\tau,
\end{align*}
which gives \eqref{eq cor1} after exponentiating the both sides.
\end{proof}

\begin{corollary}\label{cor2}
Under the same assumptions as in Theorem \ref{harnack PME}, for any points $x_1$,$ x_2\in M$, $0<t_1<t_2\leq T$ and $b\geq 1$, we have
\begin{equation}\label{eq cor2}
v(x_2, t_2)-v(x_1, t_1)\geq -\frac{d}{b}v_{max}\ln \frac{t_2}{t_1} - (\frac{b-1}{b}+\frac{|b-2|}{b}C_0v_{max})R_{max}(t_2-t_1)-\frac{b}{4}\frac{d_{t_1}^2(x_1,x_2)}{t_2-t_1},
\end{equation}
where $d$, $C_0$ and $R_{max}$ are the constants in Theorem \ref{harnack PME} , and $v_{max}=\sup_{M\times[0,T]}v$.
\end{corollary}

\begin{proof}
Let $\gamma(\tau)$, $\tau\in[t_1, t_2]$, be a geodesic from $x_1$ to $x_2$ at time $t_1$ with constant speed $\frac{d_{t_1}(x_1,x_2)}{t_2-t_1}$. Then,
\begin{align*}
v(x_2, t_2)-v(x_1,t_1)&=\int_{t_1}^{t_2}\frac{d}{d\tau}v(\gamma(\tau),\tau)d\tau\\
&=\int_{t_1}^{t_2}(v_{\tau}+(\nabla v, \gamma_{\tau}))d\tau\\
&\geq\int_{t_1}^{t_2}(v_{\tau}-\frac{1}{b}|\nabla v|^2 - \frac{b}{4}\left|\frac{d\gamma}{d\tau}\right|^2_{g_{ij}(\tau)})d\tau.
\end{align*}
Moreover, since the Ricci curvature is positive, the evolving metric is shrinking as time $\tau$ increasing. This fact together with Theorem \ref{refined noncompact harnack p>1 b>2} implies that
\begin{align*}
v(x_2, t_2)-v(x_1,t_1)&\geq\int_{t_1}^{t_2}(-\frac{b-1}{b}R-\frac{d}{b\tau}v -\frac{|b-2|}{b}C_0R_{max}v- \frac{b}{4}\left|\frac{d\gamma}{d\tau}\right|^2_{g_{ij}(t_1)})d\tau.
\end{align*}
\end{proof}

Obviously, by taking $b=2$ in Corollaries \ref{cor1} and \ref{cor2}, we have Corollaries \ref{cor 1.1} and \ref{cor 1.2}, respectively. Moreover, under the assumptions and notations as in Theorem \ref{noncompact harnack2 p>1 b=1}, one can show that the point-wise Harnack inequalities of $v$ generated by Theorem \ref{noncompact harnack2 p>1 b=1} can take the forms of \eqref{eq cor1} and \eqref{eq cor2} with $b=1$.

\bigskip
\bibliography{mybib}

\begin{thebibliography}{10}

\bibitem{ArBe1979}
Donald~G. Aronson and Philippe B{\'e}nilan.
\newblock R\'egularit\'e des solutions de l'\'equation des milieux poreux dans
  {${\bf R}^{N}$}.
\newblock {\em C. R. Acad. Sci. Paris S\'er. A-B}, 288(2):A103--A105, 1979.

\bibitem{Brendle}
Simon Brendle.
\newblock A generalization of hamilton's differential harnack inequality for
  the ricci flow.
\newblock {\em J. Differential Geom.}, 82(1):207--227, 2009.

\bibitem{Cao1992}
Huai~Dong Cao.
\newblock On {H}arnack's inequalities for the {K}\"ahler-{R}icci flow.
\newblock {\em Invent. Math.}, 109(2):247--263, 1992.

\bibitem{CaNi2005}
Huai-Dong Cao and Lei Ni.
\newblock Matrix {L}i-{Y}au-{H}amilton estimates for the heat equation on
  {K}\"ahler manifolds.
\newblock {\em Math. Ann.}, 331(4):795--807, 2005.

\bibitem{CaZh2006}
Huai-Dong Cao and Xi-Ping Zhu.
\newblock A complete proof of the {P}oincar\'e and geometrization
  conjectures---application of the {H}amilton-{P}erelman theory of the {R}icci
  flow.
\newblock {\em Asian J. Math.}, 10(2):165--492, 2006.

\bibitem{CaHa2009}
Xiaodong Cao and Richard~S. Hamilton.
\newblock Differential {H}arnack estimates for time-dependent heat equations
  with potentials.
\newblock {\em Geom. Funct. Anal.}, 19(4):989--1000, 2009.

\bibitem{Cho1991}
Bennett Chow.
\newblock The {R}icci flow on the {$2$}-sphere.
\newblock {\em J. Differential Geom.}, 33(2):325--334, 1991.

\bibitem{CLN2006}
Bennett Chow, Peng Lu, and Lei Ni.
\newblock {\em Hamilton's {R}icci flow}, volume~77 of {\em Graduate Studies in
  Mathematics}.
\newblock American Mathematical Society, Providence, RI, 2006.

\bibitem{DH}
Panagiota Daskalopoulos and Richard~S. Hamilton.
\newblock Regularity of the free boundary for the porous medium equation.
\newblock {\em J. Amer. Math. Soc.}, 11(4):899--965, 1998.

\bibitem{Ham1988}
Richard~S. Hamilton.
\newblock The {R}icci flow on surfaces.
\newblock In {\em Mathematics and general relativity ({S}anta {C}ruz, {CA},
  1986)}, volume~71 of {\em Contemp. Math.}, pages 237--262. Amer. Math. Soc.,
  Providence, RI, 1988.

\bibitem{Ham1993jdg}
Richard~S. Hamilton.
\newblock The {H}arnack estimate for the {R}icci flow.
\newblock {\em J. Differential Geom.}, 37(1):225--243, 1993.

\bibitem{Ham1993cag}
Richard~S. Hamilton.
\newblock A matrix {H}arnack estimate for the heat equation.
\newblock {\em Comm. Anal. Geom.}, 1(1):113--126, 1993.

\bibitem{HHL2013}
Guangyue Huang, Zhijie Huang, and Haizhong Li.
\newblock Gradient {E}stimates for the {P}orous {M}edium {E}quations on
  {R}iemannian {M}anifolds.
\newblock {\em J. Geom. Anal.}, 23(4):1851--1875, 2013.

\bibitem{LiYa1986}
Peter Li and Shing-Tung Yau.
\newblock On the parabolic kernel of the {S}chr\"odinger operator.
\newblock {\em Acta Math.}, 156(3-4):153--201, 1986.

\bibitem{LNVV2009}
Peng Lu, Lei Ni, Juan-Luis V{\'a}zquez, and C{\'e}dric Villani.
\newblock Local {A}ronson-{B}\'enilan estimates and entropy formulae for porous
  medium and fast diffusion equations on manifolds.
\newblock {\em J. Math. Pures Appl. (9)}, 91(1):1--19, 2009.

\bibitem{Ni2006}
Lei Ni.
\newblock A note on {P}erelman's {LYH}-type inequality.
\newblock {\em Comm. Anal. Geom.}, 14(5):883--905, 2006.

\bibitem{Ni2007}
Lei Ni.
\newblock A matrix {L}i-{Y}au-{H}amilton estimate for {K}\"ahler-{R}icci flow.
\newblock {\em J. Differential Geom.}, 75(2):303--358, 2007.

\bibitem{Per2002a}
Grisha Perelman.
\newblock The entropy formula for the {R}icci flow and its geometric
  applications.
\newblock {\em arXiv:math/0211159}, 2002.

\bibitem{ScYa1994}
R.~Schoen and S.-T. Yau.
\newblock {\em Lectures on differential geometry}.
\newblock Conference Proceedings and Lecture Notes in Geometry and Topology, I.
  International Press, Cambridge, MA, 1994.
\newblock Lecture notes prepared by Wei Yue Ding, Kung Ching Chang [Gong Qing
  Zhang], Jia Qing Zhong and Yi Chao Xu, Translated from the Chinese by Ding
  and S. Y. Cheng, Preface translated from the Chinese by Kaising Tso.

\bibitem{Vaz2007}
Juan~Luis V{\'a}zquez.
\newblock {\em The porous medium equation}.
\newblock Oxford Mathematical Monographs. The Clarendon Press Oxford University
  Press, Oxford, 2007.
\newblock Mathematical theory.

\end{thebibliography}
\bibliographystyle{plain}

\end{document}